\documentclass[11pt]{amsart}
\usepackage{color}
\usepackage{listings}
\usepackage{float}
\floatstyle{ruled}
\usepackage{caption}
\usepackage{amsmath}
\usepackage{amssymb}
\usepackage{bbm}
\usepackage{enumerate}
\restylefloat{figure}
\newtheorem{thm}{Theorem}[section]
\newtheorem{cor}[thm]{Corollary}

\newtheorem{prop}[thm]{Proposition}

\theoremstyle{definition}

\newtheorem{rem}[thm]{Remark}

\numberwithin{equation}{section}

\makeatletter
\@ifpackageloaded{showkeys}{\oddsidemargin=3cm\evensidemargin=3cm}{\oddsidemargin=0.7cm\evensidemargin=0.7cm}\makeatother
\newcommand{\lin}{\operatorname{span}}
\newcommand{\supp}{\operatorname{supp}}

\newcommand{\dif}{\,\mathrm{d}}

\newcommand{\charfun}{\ensuremath{\mathbbm 1}}

\allowdisplaybreaks

\begin{document}
\title{Orthoprojectors on perturbations of splines spaces}
\author[K. Keryan]{Karen Keryan}
\address{Yerevan State University, Alex Manoogian 1, 0025 Yerevan, Armenia,
	American University of Armenia, Marshal Baghramyan 40, 0019 Yerevan, Armenia}
\email{karenkeryan@ysu.am, kkeryan@aua.am}

\author[M. Passenbrunner]{Markus Passenbrunner}
\address{Institute of Analysis, Johannes Kepler University Linz, Austria, 4040 Linz, Altenberger Strasse 69}
\email{markus.passenbrunner@jku.at}


\begin{abstract}
	We show that $L^\infty$-norms of orthoprojectors on  certain
types of perturbations of spline spaces are bounded independently of the knot
sequence. Explicit
applications of this result are given, one of them being orthoprojectors onto Chebyshevian
spline spaces.
\end{abstract}
\maketitle
\section{Introduction}

In this paper, we extend Shadrin's theorem \cite{Shadrin2001} on the  
boundedness of the polynomial spline orthoprojector on $L^\infty$  by a constant 
that does not depend on the underlying  univariate grid 
 to  certain perturbations of spline spaces.

One of the main reasons to consider this extension of spline
orthoprojectors is that in recent years, it turned out that in many cases
 (see e.g. \cite{Shadrin2001, PassenbrunnerShadrin2014,
Passenbrunner2014, MuellerPassenbrunner2017,GevorkyanKeryanPoghosyan2020,Passenbrunner2020b,Passenbrunner2020a}),
sequences of   orthogonal projections onto classical spline spaces  corresponding to arbitrary 
grid sequences  
behave like sequences of conditional
expectations (or, more generally, like  martingales) and we
want to extend martingale type
results to an even larger class of orthogonal projections.

In order to explain those martingale type results, we have to introduce a little bit of
terminology: Let $k$ be a positive integer, $(\mathcal F_n)$ an increasing
sequence of interval $\sigma$-algebras of sets in $[0,1]$, where we say that a
$\sigma$-algebra is an \emph{interval
$\sigma$-algebra} if it is generated 
by a finite partition of $[0,1]$ into intervals of positive length. Moreover,
let
\[
	S_n^{(k)} = \{f\in C^{k-2}[0,1] : f\text{ is a polynomial of order $k$ on
	each atom of $\mathcal F_n$} \}
\]
be the spline space of order $k$ corresponding to $\mathcal F$
and  define $P_n^{(k)}$ as the orthogonal projection operator onto
$S_n^{(k)}$ with respect to the
$L^2$  inner product on $[0,1]$  with Lebesgue measure  $|\cdot|$.  The space $S_n^{(1)}$ 
(interpreting $C^{-1}[0,1]$ as the space of all real-valued functions on
$[0,1]$)
consists of
piecewise constant functions and $P_n^{(1)}$ is the conditional expectation
operator with respect to the $\sigma$-algebra $\mathcal F_n$.
 Similarly to the definition of martingales, we introduce the following
notion pertaining to spline spaces:
let $(f_n)_{n\geq 0}$ be a sequence of integrable functions, we call this
sequence a \emph{$k$-martingale spline sequence (adapted to $(\mathcal F_n)$)}, if
\[
	P_n^{(k)} f_{n+1} = f_n,\qquad n\geq 0.
\]

 Classical martingale theorems such as Doob's
	inequality, the martingale convergence theorem or Burkholder's
	inequality in fact carry over to
	$k$-martingale spline sequences corresponding to  \emph{arbitrary}
	filtrations ($\mathcal F_n$) of the above type. Indeed, we have
	for any positive integer $k$
		\begin{enumerate}[(i)]
			\item\label{it:splines1} (Shadrin's theorem) 
				there exists
				a constant $c_k$ depending only on $k$ such that 
				\[
					\sup_n\| P_n^{(k)} : L^1 \to L^1 \| \leq
					c_k,
				\]
			\item \label{it:splines2}
				there exists a constant $c_k$  depending only on
				$k$ such that for any $k$-martingale
				spline sequence
				$(f_n)$ and any
				$\lambda>0$, 
				\begin{equation*}
					|\{ \sup_n |f_n| > \lambda \}| \leq c_k
				\frac{\sup_n\|f_n\|_{L^1}}{
				\lambda},
				\end{equation*}

			\item\label{it:splines3} for all $p\in (1,\infty]$ there exists a constant
				$c_{p,k}$  depending only on $p$ and
				$k$ such that for all $k$-martingale
				spline sequences
				$(f_n)$,
				\begin{equation*}
					\big\| \sup_n |f_n| \big\|_{L^p}
					\leq c_{p,k}
					\sup_n\|f_n\|_{L^p},\ 
				\end{equation*}
			\item\label{it:splines4}
				if $(f_n)$ is an
				$L^1$-bounded $k$-martingale spline sequence, then
				$(f_n)$ converges
				 almost surely to some $L^1$-function.

			\item	for all $p\in(1,\infty)$,
				scalar-valued $k$-spline-differences
				converge unconditionally in $L^p$, i.e. for all
				$f\in L^p$,
				\begin{equation*}
					\big\|\sum_n \pm (P_n^{(k)} -
					P_{n-1}^{(k)})f \big\|_{L^p}
					\leq c_{p,k} \|f\|_{L^p},
				\end{equation*}
				for some constant $c_{p,k}$ depending only on
				$p$ and $k$.
		\end{enumerate}
	Item \eqref{it:splines1} is proved in \cite{Shadrin2001}, for a considerably
	shorter proof we refer to \cite{Golitschek2014}.
	Banach space valued versions of \eqref{it:splines2}--\eqref{it:splines4}
	are proved  in \cite{PassenbrunnerShadrin2014,
				MuellerPassenbrunner2017} and (v) is proved in
				\cite{Passenbrunner2014}.
				For periodic spline spaces some of those properties are proved in
				\cite{Passenbrunner2017,KeryanPassenbrunner2017}.
 The basic starting point in proving the results (ii)--(v) independently
of the filtration $(\mathcal F_n)$ is Shadrin's theorem (i) and in this paper we
prove its analogue for certain perturbations of spline spaces.

	An important tool in the analysis of the operators $P_n^{(k)}$ as well
	as in the formulation of our perturbation result in Section
	\ref{sec:perturb} are special localized bases of the spaces
	$S_n^{(k)}$ and perturbations thereof. In the case of the space
	$S_n^{(k)}$ this is the so called B-spline basis $(M_i)$, normalized in
	$L^1$.
	By definition, the spaces $(S_n^{(k)})_n$ are nested, i.e.,
	$S_n^{(k)} \subset S_{n+1}^{(k)}$ for all $n$. Moreover, the 
	B-splines $(M_i)$ have the following properties:
	\begin{enumerate}[(a)]
		\item $\supp M_i \cap \supp M_j =\emptyset$ for $|i-j|\geq k$,
		\item \label{it:local} each function $M_j$ only depends on the  
			local form of $\mathcal F_n$, i.e., on $\mathcal F_n\cap \supp
			M_j$,
		\item \label{it:Misupp}$\supp M_{i}$ is a union of atoms of $\mathcal F_n$.
	\end{enumerate}

	Note that	the uniform (in $n$) $L^1$-boundedness of $P_n^{(k)}$ stated in (i) 
	and their uniform $L^\infty$-boundedness are equivalent, as $P_n^{(k)}$ 
		is   
		self-adjoint with respect to the inner product  $\langle f,g\rangle =
	\int_0^1 f(x)g(x)\dif x$ since it is an orthogonal projection.
	Defining the renormalized B-spline function $N_i = (|\supp M_i|/k) M_i$,
	the uniform boundedness of
	$\|P_n^{(k)}\|_{L^\infty}=\|P_n^{(k)}:L^\infty\to L^\infty\|$ can be rephrased
	in terms of the Gram matrix 
	$G = (\langle M_i, N_j\rangle)$. In fact, the uniform boundedness
	 of
	$\|P_n^{(k)}\|_{L^\infty}$
	is equivalent (see \cite{Ciesielski2000} or
	\cite{PassenbrunnerShadrin2014}) to the uniform estimate 
	\begin{equation}\label{eq:graminverse}
		\|G^{-1}\|_\infty\leq C_k,
	\end{equation}
	where $C_k$ is some constant depending only on the spline order $k$.

\section{Projectors onto perturbed spline spaces}\label{sec:perturb}
In this section we define what we mean by perturbations of spline spaces
	and prove the corresponding theorem about the uniform boundedness on
	$L^\infty$ of the
	associated orthoprojectors.

	Let $k$ be an arbitrary positive integer, $\mu$ be a non-atomic probability measure on $[0,1]$ and $\theta :
	[0,\infty)\to [0,\infty)$ be an increasing function with $\lim_{t\to 0} \theta(t)=0$.
For any interval $\sigma$-algebra $\mathcal F$ on $[0,1]$,  let $S_{\mathcal F}^{(k)}$ be the spline
space of order $k$ corresponding to $\mathcal F$ and let $S_{\mathcal F,p}^{(k)}\subset L^2(\mu)$ be a finite
dimensional linear space.
As above, we denote by $(M_i)$ the B-spline basis of $S_{\mathcal F}^{(k)}$ and
additionally, we use the notation $\langle f,g\rangle_\mu = \int_0^1 f(x) g(x)
\dif \mu(x)$ and $|\mathcal F|_\mu = \max_A \mu(A)$, where $\max$ is taken over all
atoms $A$ of $\mathcal F$. We also define the \emph{$\mathcal F$-support}
$\supp_{\mathcal F} f$ of a function $f:[0,1]\to\mathbb R$ to be the smallest
subset of $[0,1]$ that is a union of atoms of $\mathcal F$ and contains the
support $\supp f$ of $f$.

We say that the collection $(S_{\mathcal F,p}^{(k)})_{\mathcal F}$ is a
\emph{($\mu$,$\theta)$-perturbation of the spline spaces $(S_{\mathcal F}^{(k)})_{\mathcal F}$ with constant
$C$} if, for any $\mathcal F$, 
$S_{\mathcal F,p}^{(k)}$ admits a basis  $(M_i^p)$ so that 
for any indices $i,j$, we have 
\begin{enumerate}
	\item \label{it:perturb}
		$\big| {\mu(\supp_{\mathcal F} M_{j}^p)} { \langle M_i^p, M_j^p\rangle_\mu}
		- {|\supp_{\mathcal F} M_j|} {\langle
				M_i,M_j\rangle} \big| \leq \theta(|\mathcal F|_\mu)$,
			\item \label{it:supp}$\supp_{\mathcal F} M_i^p \cap
				\supp_{\mathcal F} M_j^p =\emptyset$ for $|i-j|\geq C$,
	\item \label{it:norm}$  \|M_i^p\|_{L^\infty(\mu)}  \cdot
		\mu(\supp_{\mathcal F} M_{i}^p)
		\leq C$.
	\end{enumerate}

	If $(S_{\mathcal F,p}^{(k)})_{\mathcal F}$ is a
($\mu$,$\theta)$-perturbation of $(S_{\mathcal F}^{(k)})_{\mathcal F}$, we say that the spaces $(S_{\mathcal F,p}^{(k)})_{\mathcal F}$ are
\emph{compatible} if for each $\mathcal G\subset \mathcal F$, the following
conditions are satisfied: 
\begin{enumerate}\setcounter{enumi}{3}
	\item Nestedness: $S_{\mathcal G,p}^{(k)} \subset S_{\mathcal
		F,p}^{(k)}$,
	\item  \label{it:localstructure} Local structure of the basis:
		for each set $I\subset [0,1]$ so that the trace
			$\sigma$-algebras $\mathcal F\cap I$ and  $\mathcal
			G\cap I$ coincide and
			each basis function $M_{i}^p$ of $S_{\mathcal
			F,p}^{(k)}$ with 
			$\supp_{\mathcal F} M_i^p\subset I$, we also have $M_i^p\in
			S_{\mathcal G,p}^{(k)}.$
\end{enumerate}

If $(S_{\mathcal F,p}^{(k)})_{\mathcal F}$ is a
($\mu$,$\theta)$-perturbation of $(S_{\mathcal F}^{(k)})_{\mathcal F}$, 
denote by $P_{\mathcal F,\mu}:L^2(\mu)\to L^2(\mu)$ the orthogonal projection
operator onto the space $S_{\mathcal F,p}^{(k)}$ with respect to the inner
product $\langle\cdot,\cdot\rangle_\mu$.

\begin{rem}
	For B-spline functions $(M_i)$ that form a basis of some spline space $S_{\mathcal
	F}^{(k)}$, the notions of support and $\mathcal
	F$-support coincide by property \eqref{it:Misupp} on page
	\pageref{it:Misupp}, i.e. we have $\supp M_i = \supp_{\mathcal F} M_i$
	for any index $i$.

 The collection of spline spaces $(S_{\mathcal F}^{(k)})_{\mathcal F}$
is a compatible $(|\cdot|,0)$-perturbation of itself.

The last condition \eqref{it:localstructure} means that the basis
			function $M_i^p$ is, in some sense, determined only by
			the local structure of $\mathcal F$. 
			Observe that by their very definition and  property
				\eqref{it:local} on page \pageref{it:local}
			of the B-spline functions $(M_i)$, the spline spaces $(S_{\mathcal
			F}^{(k)})_{\mathcal F}$ are
			compatible.

 If $\mu$ is an arbitrary non-atomic finite measure on
$[0,1]$, we can define the probability measure $\overline\mu = \mu/m$ with
$m=\mu[0,1]$. Additionally the functions 
${\overline{M_i}^p}
= m M_i^p$ and $\overline{\theta}(t) = \theta(mt)$ satisfy conditions
\eqref{it:perturb}--\eqref{it:norm} as well as $P_{\mathcal F,\overline{\mu}} =
P_{\mathcal F,\mu}$. Therefore, there is no loss of generality in assuming  
$\mu$ to be  a probability measure.
		\end{rem}

\begin{thm}\label{thm:perturbation}
	Let $k$ be a positive integer, $\mu$ be a non-atomic probability measure on
	$[0,1]$,
	$\theta:[0,\infty)\to [0,\infty)$ be an increasing function with
		$\lim_{t\to 0} \theta(t)=0$ and $C$ be a positive
	constant. Assume that 
	$(S_{\mathcal F,p}^{(k)})_{\mathcal F}$ is a $(\mu,\theta)$-perturbation of $(S_{\mathcal
	F}^{(k)})_{\mathcal F}$ with constant $C$.

	Then, there exists a constant $K_1$ depending only on $C$ and $k$ so that
	\[
		\sup_{\mathcal F} \| P_{\mathcal F,\mu} : L^\infty(\mu) \to
		L^\infty(\mu)\| \leq K_1,
	\]
	where $\sup$ is taken over all interval $\sigma$-algebras $\mathcal F$ with
	$|\mathcal F|_\mu \leq \varepsilon$ for $\varepsilon>0$ taken so that
	$\theta(\varepsilon)\leq k/(4C_kC)$
	with the constant  $C_k$ from \eqref{eq:graminverse}.

	Additionally, if the spaces $(S_{\mathcal F,p}^{(k)})_{\mathcal F}$ are
	compatible
	we have
	\[
		\sup_{\mathcal F} \| P_{\mathcal F,\mu} : L^\infty(\mu) \to
		L^\infty(\mu)\| \leq K_2,
	\]
	where $\sup$ is taken over all interval $\sigma$-algebras $\mathcal F$
	and $K_2$ is a constant depending only on $C$, $\theta $ and $k$.
\end{thm}
\begin{proof}
	Let $\mathcal F$ be an arbitrary interval $\sigma$-algebra with
	$|\mathcal F|_\mu \leq \varepsilon$ and $(M_i)$, $(M_i^p)$ be the bases
	corresponding to the spaces $S_{\mathcal F}^{(k)}$ and $S_{\mathcal
	F,p}^{(k)}$, respectively, satisfying conditions
	\eqref{it:perturb}--\eqref{it:norm}.
Let $G=(\langle M_{i},N_{j}\rangle)_{ij}$ and $G_p=(\langle
M^p_{i},N^p_{j}\rangle_\mu)_{ij}$, where $N_j= (|\supp M_j|/k) M_j$
denotes the re-normalized
classical B-spline and 
\begin{equation}\label{eq:defNp}
	N^p_{j}:=(\mu(\supp_{\mathcal F} M_{j}^p)/k)\cdot M^p_{j}.
\end{equation}
  By Shadrin's theorem in the form of
\eqref{eq:graminverse}, there exists a constant $C_k$ depending only on $k$ such that
\begin{equation}\label{eq:G G^-1}
\|G^{-1}\|_\infty\leq C_k.
\end{equation} 
Taking into account that both $G$ and $G_p$ are   $C$-banded matrices by
\eqref{it:supp} we get by property \eqref{it:perturb} that
\begin{equation*}
	\|G_p-G\|_\infty \leq 2C \theta(|\mathcal F|_\mu)/k.
\end{equation*}
Denote $X= -G^{-1}(G_p-G).$ 
Since $|\mathcal F|_\mu\leq \varepsilon$,
with $\varepsilon$ so that $\theta(\varepsilon)\leq k/(4C_k C)$, we obtain
$$\|X\|_\infty\leq \|G^{-1}\|_\infty \|G_p-G\|_\infty \leq 1/2.$$ 
Thus we have $(I-X)^{-1}=\sum_{k=0}^\infty X^k$ and 
\begin{equation}\label{eq:neumann}
	\|(I-X)^{-1}\|_\infty\leq \sum_{k=0}^\infty \|X\|_\infty^k\leq 2.
\end{equation}
Notice that $(G_p)^{-1}=(G+(G_p-G))^{-1}=(I+G^{-1}(G_p-G))^{-1}G^{-1}=(I-X)^{-1}G^{-1}.$ 
Using estimate \eqref{eq:neumann} together with \eqref{eq:G G^-1} we obtain 
\begin{equation}\label{eq:G_p^-1}
\|(G_p)^{-1}\|_\infty\leq 2 \|G^{-1}\|_\infty\leq 2C_k.
\end{equation}
Additionally, we observe that we have a similar bound also for the norm
$\|G_p\|_\infty$ of the banded matrix $G_p$
by properties \eqref{it:supp} and \eqref{it:norm}: 
\begin{equation*}
	\begin{aligned}
	\|G_p\|_\infty &= \max_i \sum_j |\langle M_i^p, N_j^p\rangle_\mu|  \\
	&\leq \max_i \sum_{j : \supp M_i^p\cap \supp M_j^p\neq \emptyset}
	\|M_i^p\|_{L^1(\mu)} \|N_j^p\|_{L^\infty(\mu)} \\
	&\leq 2C\max_{i,j}  
	\|M_i^p\|_{L^\infty(\mu)}\mu(\supp M_i^p)\cdot \frac{\mu(\supp M_j^p)}{k} 
	\|M_j^p\|_{L^\infty(\mu)} \leq 2C^3/k.
\end{aligned}
\end{equation*}
Next, we apply Demko's theorem \cite{Demko1977} that states--in particular--that
the boundedness of $\|G_p\|_\infty$ and $\|(G_p)^{-1}\|_\infty$ of a banded
matrix $G_p$
is sufficient to deduce the geometric decay of the matrix
$(a_{ij})_{ij}=G_p^{-1}$, i.e. we have for any $i,j$ the  estimate $|a_{ij}|\leq c q^{|i-j|}$, 
where $c$ and $q\in(0,1)$ are two constants
depending only on $\|G_p\|_\infty$ and $\|(G_p)^{-1}\|_\infty$.  In our
case this means that $c$ and $q$ only depend on $C$ and $k$.
Using this matrix, the projection operator $P_{\mathcal F,\mu}$ onto
$S_{\mathcal F,p}$ is given by the formula
 \[
	 P_{\mathcal F,\mu}f(t)=\sum_{i,j} a_{ij}\langle f,M_j^p\rangle_\mu N^p_i(t).
 \]
 Hence, by \eqref{it:norm} and the geometric decay estimate for $a_{ij}$, we
 have for $\mu$-almost-every $t\in[0,1]$ the inequality 
\begin{align*}|Pf(t)|
	& \leq \sum_{i,j: t\in\supp N_i^p} |a_{ij}| \|f\|_{L^\infty(\mu)} \|M_j^p\|_{L^1(\mu)}
	\|N_i^p\|_{L^\infty(\mu)} \\
	&\leq \frac{c C^2}{k} \sum_{i,j:t\in\supp N_i^p} q^{|i-j|} \| f
	\|_{L^\infty(\mu)}.
 \end{align*}
 Now we use again property \eqref{it:supp} to deduce 
 \[
	 \|Pf\|_{L^\infty(\mu)} \leq K_1
	 \|f\|_{L^\infty(\mu)}
 \]
 with $K_1=4cC^3\sum_{j=0}^\infty q^{j}/k$ depending only on $C$ and $k$. This
 concludes the proof of the first assertion of the theorem.

 Next, we assume that the spaces $(S_{\mathcal F,p}^{(k)})_{\mathcal F}$ are
 compatible as well.
	 Let $\mathcal G$ be an interval $\sigma$-algebra with
 $|\mathcal G|_\mu>\varepsilon$.
 Let $U\subset [0,1]$ be the point set consisting of all atoms $A$ of $\mathcal G$
 with $\mu(A)\leq \varepsilon$.
 Next, we let $\mathcal F$ be an
 interval $\sigma$-algebra with $|\mathcal F|_\mu\leq \varepsilon$ that is a
 refinement of $\mathcal G$ and coincides with $\mathcal G$ on $U$ and has the
 property that for any atom $A\subset U^c$ of $\mathcal F$, we have $\mu(A)\geq
 \varepsilon/2$. This is possible since $\mu$ is assumed to be non-atomic.
 Let $N_i^p$ be a basis function  from $S_{\mathcal F,p}^{(k)}$, given by
	 \eqref{eq:defNp}, with $\supp_{\mathcal F} N_i^p \subset U$.
 Then, since $S_{\mathcal F,p}^{(k)}$ and $S_{\mathcal G,p}^{(k)}$ are
 compatible, the function $N_{i}^p$ is also contained in $S_{\mathcal
 G,p}^{(k)}$. Since the operators $P_{\mathcal G,\mu}$ and $P_{\mathcal F,\mu}$
 are both orthogonal projections, we get that both $P_{\mathcal G,\mu} f - f$
 and $P_{\mathcal F,\mu} f - f$ are orthogonal to
 the function $N_i^p$ in $L^2(\mu)$.  This implies
 \[
	\langle (P_{\mathcal F,\mu} - P_{\mathcal G,\mu}) f, N_i^p \rangle_\mu =
	0,\qquad \supp_{\mathcal F} N_i^p \subset U.
 \]
 Therefore, since $S_{\mathcal G,p}^{(k)}\subset S_{\mathcal F,p}^{(k)}$, we
 expand 
 \begin{equation}\label{eq:differenceP}
	 (P_{\mathcal F,\mu} - P_{\mathcal G,\mu}) f = \sum_{i :
		 \supp_{\mathcal F} N_i^p
	\not\subset U} d_i N_i^{p,*},
\end{equation}
 where $(N_j^{p,*})$ denotes the basis of $S_{\mathcal F,p}^{(k)}$ that is dual
 to the basis $(N_j^p)$ w.r.t the inner product in $L^2(\mu)$. Moreover, the
 coefficients $d_i$ and the functions $N_i^{p,*}$ are given by the formulas
 \[
d_i = \langle (P_{\mathcal F,\mu} - P_{\mathcal G,\mu}) f, 
	N_i^p \rangle_\mu,  \qquad
	 N_i^{p,*} = \sum_j a_{ij} M_j^p=\sum_j b_{ij} N_j^p,
 \]
 where the matrix $(a_{ij})$, as above, denotes the inverse $G_p^{-1}$ of the
 Gram matrix $G_p =
 (\langle M_i^p, N_j^p\rangle_\mu)$ and the matrix $(b_{ij})$ denotes the
 inverse of the matrix $(\langle N_i^p, N_j^p\rangle_\mu)$. Since $N_i^p = a_i
 M_i^p$ with $a_i:=\mu(\supp_{\mathcal F}  N_i^p)/k$, the coefficients $a_{ij}$ and $b_{ij}$ are related by the equation
 $b_{ij} = a_{ij}/a_j$. Since the matrix $(b_{ij})$ is symmetric, we also have
 $b_{ij}=b_{ji}= a_{ji}/a_i$. The geometric decay of the matrix $(a_{ij})$
 and (3) imply the pointwise estimate
 \[
	 |N_i^{p,*}(t)| \leq \sum_{j} |b_{ij}| |N_j^p(t)| \leq
	 \frac{cC}{a_i} \sum_{j: t\in\supp_{\mathcal F} N_j^p } q^{|i-j|},\qquad \text{for
	 $\mu$-a.e. $t$}.
 \]
 Denoting by $j(t)$ any fixed index with $t\in\supp_{\mathcal F} N_{j(t)}^p$, we obtain  the estimate
 \begin{equation}\label{eq:estN*}
	 |N_i^{p,*}(t)| \leq \frac{c_1}{\mu(\supp_{\mathcal F}  N_i^p)} q^{|i - j(t)|},
\end{equation}
with 
$c_1 := 4kcC^2q^{-C} \sum_{\ell=0}^\infty q^\ell$.
 Next, we estimate the coefficients $d_i$ from equation \eqref{eq:differenceP}.
	 Since $P_{\mathcal F,\mu}$ and $P_{\mathcal G,\mu}$ are both orthogonal
 projections, they have an $L^2(\mu)$-norm of $1$. Thus, we estimate $d_i$ as
 \begin{equation}\label{eq:dj}
	|d_i| \leq 2 \| f \|_{L^2(\mu)} \|N_i^p\|_{L^2(\mu)} \leq
	\frac{2C}{k}\|f\|_{L^\infty(\mu)} \mu(\supp_{\mathcal F} N_i^p)^{1/2},
\end{equation}
 where in the last step we used property \eqref{it:norm}. 

Now, insert estimates \eqref{eq:estN*} and \eqref{eq:dj} into
\eqref{eq:differenceP} and observe that for $i$ such that $\supp_{\mathcal F} N_i^p
\not\subset U$, we have $\mu(\supp_{\mathcal F} N_i^p)\geq \varepsilon/2$ by definition
	of $\mathcal F$. 
Therefore, for $\mu$-almost-every $t\in[0,1]$,
\begin{align*}
	|(P_{\mathcal F,\mu} - P_{\mathcal G,\mu}) f(t)| &= \Big|\sum_{i :
		\supp_{\mathcal F} N_i^p
	\not\subset U} d_i N_i^{p,*}(t)\Big| \\
	& \leq 4c_1 
	\frac{C}{k\varepsilon^{1/2}}\sum_{i : \supp_{\mathcal F} N_i^p
	\not\subset U} \|f\|_{L^\infty(\mu)} q^{|i-j(t)|} \leq c_3 
	\|f\|_{L^\infty(\mu)}
\end{align*}
with $c_3:= 8c_1 C\sum_{j=0}^\infty q^j/\varepsilon^{1/2}$. Taking the
supremum
over $t\in[0,1]$, we obtain that the operator $P_{\mathcal
F,\mu}-P_{\mathcal G,\mu}:L^\infty(\mu)\to L^\infty(\mu)$ is bounded by $c_3$.
Therefore this and the first part of the Theorem \ref{thm:perturbation} imply
\[
	\|P_{\mathcal G,\mu}\|_{L^\infty(\mu)}\leq \|P_{\mathcal G,\mu} -
	P_{\mathcal F,\mu}\|_{L^\infty(\mu)} + \| P_{\mathcal F,\mu}
	\|_{L^\infty(\mu)}  \leq c_3+K_1 
\]
and thus we get the second assertion of the theorem with $K_2 := c_3+ K_1$.
 \end{proof}

\section{Applications}\label{sec:cheb}
In this section we investigate concrete applications of Theorem
\ref{thm:perturbation}. The first example considers projection operators on
classical spline spaces corresponding to different measures and the second
example considers orthoprojectors onto Chebyshevian spline spaces. 

\subsection{Weighted spline spaces}
	We consider the setting of standard B-splines $(M_j)$ on an
	interval $\sigma$-algebra $\mathcal F$, $S_{\mathcal F} = \lin (M_j)$. 
	Moreover, consider the measure $\mu$ with $\dif \mu = w \dif x$,   $w:[0,1]\to
	(0,\infty)$ being a continuous function satisfying the inequalities $M^{-1} \leq w\leq M$ 
	for some constant $M$. Set $M_i^p = M_i/ w(c_i)$
	with the center $c_i$ of $\supp M_i$. 
	We now show properties (1)--(3) on page \pageref{it:perturb} in this
	setting.
	First consider property (1): we see that 
	\begin{align*}
			 |{\mu(\supp M_{j}^p)} & 
			{ \langle M_i^p, M_j^p\rangle_\mu} - {|\supp M_j|} {\langle
				M_i,M_j\rangle}|  \\
				&\leq \int\limits_{\supp M_i\cap\, \supp M_j}
		M_i(x) M_j(x) |\supp M_j|  \Big| \frac{\mu(\supp M_{j})}{|\supp M_j|} 
		\frac{w(x)}{w(c_i)w(c_j)} - 1\Big| \dif x. 
		\end{align*}
	Note that  $ {\mu(\supp M_{j})}/{|\supp M_j|}=w(\xi_j)$ for some
	$\xi_j\in \supp M_j. $ Moreover, for $x\in \supp M_i\cap \supp M_j$, 
	\begin{align*}
		\left|\frac{w(\xi_j)w(x)}{w(c_i)w(c_j)} - 1\right| &\leq
	M^2\big(w(\xi_j)|w(x)-w(c_j)|+w(c_j)|w(\xi_j)-w(c_i)|\big) \\
	&\leq 2M^3 \omega(w, k|\mathcal F|),
\end{align*}
where $\omega(f,\delta):= \sup_{|x-y|<\delta}
	|f(x)-f(y)|$ denotes the modulus of continuity  of $f$. As
	$$\int\limits_{\supp M_i\cap \, \supp M_j}
		M_i(x) M_j(x) |\supp M_j|   \dif x \leq  (\|M_i\|_\infty|\supp M_i|)(\|M_j\|_\infty|\supp M_j|)\leq k^2$$
	and $\omega(w,k|\mathcal F|) \to 0$ as $|\mathcal F|/M\leq |\mathcal
	F|_\mu \to 0$ by the continuity of the
	function $w$, we obtain property (1). 
	 Property (2) follows from the
	corresponding properties of B-splines and  property (3) follows from
	the inequality
	\[
		 \|M_i^p\|_{L^\infty(\mu)}  \cdot \mu(\supp M_{i}^p)
		\leq M^2 \|M_i\|_{L^\infty} |\supp M_{i}|\leq kM^2.
	\]

	Thus, an application of  Theorem \ref{thm:perturbation} yields the
	following
	\begin{cor}
		Suppose that $\dif \mu = w\dif x$ for some 
		continuous function $w$ on $[0,1]$ satisfying $M^{-1}\leq w\leq
		M$ for some constant $M>0$.
		For an interval $\sigma$-algebra $\mathcal F$ and a non-negative
		integer $k$, let 
		$S_{\mathcal F}$ be the corresponding spline space of order $k$
		and $P_{\mathcal F,\mu}$ the orthogonal projection operator onto
		$S_{\mathcal F}$ with respect to the inner product $\langle
		\cdot,\cdot\rangle_\mu$.

		Then, there exists a constant $C$, depending only on $k$, $M$ and the modulus
		of continuity of $w$, so that
		\[
			\sup_{\mathcal F}\| P_{\mathcal F,\mu} : L^\infty(\mu)\to L^\infty(\mu)
			\| \leq C,
		\]
		where $\sup$ is taken over all interval $\sigma$-algebras
	$\mathcal F$.
	\end{cor}

\subsection{Chebyshevian spline spaces}\label{subsec:prel}
Here, we only give the necessary definitions and results pertaining to
Chebyshevian spline spaces  used to apply Theorem~\ref{thm:perturbation}.
 As a basic reference and for more information about Chebyshevian splines, we
 refer to the book
	\cite{Schumaker1981}, in particular Chapter~9.
Suppose that for a positive integer $k$, $w=(w_1,\ldots,w_k)$ is a vector consisting of $k$ positive
functions (weights) on $[0,1]$ with $w_i\in C^{k-i+1}[0,1]$
for all
$i\in\{1,\ldots,k\}$. 
Then define  the vector $u = (u_1,\ldots,u_k)$ of functions by 
\begin{align*}
u_1=w_1\textrm{ and }	u_i(x) &= 
	w_1(x)\int_0^x
	w_2(s_2)\cdots \int_0^{s_{i-1}} w_{i}(s_{i}) \dif s_{i} \cdots \dif s_2,
	\qquad i = 2,\ldots,k.
\end{align*}
Let $\mathcal F$ be an interval
$\sigma$-algebra  
and define
the Chebyshevian spline space $S_{\mathcal F,w}$ as 
\begin{align*}
	S_{\mathcal F,w} &= \{ f\in C^{k-2}[0,1]: f \in \lin
	\{u_1,\ldots,u_k\}
	\text{ on each atom of $\mathcal F$}\}.
\end{align*}
If we choose the constant weight functions $w_1 = \cdots  =
w_k=\operatorname{const}$, we get the
classical spline space $S_\mathcal F$ of order $k$.
Given the weights $w=(w_1,\ldots,w_k)$ and the corresponding system $u=(u_1,\ldots,
u_k)$, we define its dual system $u^*(w) =
(u_1^*,\ldots,u_k^*,u_{k+1}^*)$ by $u_1^*(x)=1$ and 
\begin{equation}\label{eq:dualU}
	u_i^*(x) = u_{0,i}^*(x)= \int_0^x w_k(s_k) \int_0^{s_k}
	w_{k-1}(s_{k-1}) \cdots
	\int_0^{s_{k-i+3}} w_{k-i+2}(s_{k-i+2}) \dif s_{k-i+2}\cdots \dif s_k
\end{equation}
for $i=2,\ldots,k+1$.
Moreover, we define for $j=0,\ldots,k$ the functions $u_{j,1}^*(x) = 1$ and
\[
	u_{j,i}^*(x) = \int_0^x w_{k-j}(s_k) \int_0^{s_k}
	w_{k-j-1}(s_{k-1}) \cdots
	\int_0^{s_{k-i+3}} w_{k-j-i+2}(s_{k-i+2}) \dif s_{k-i+2}\cdots \dif s_k
\]
for $i=2,\ldots,k+1-j$.

Next, define $h_1^w(x,y) := w_1(x)$ and 
\begin{equation}\label{def:gh}
	h_j^w(x,y) := w_1(x) \int_y^x w_2(s_2)\int_y^{s_2} \cdots
	\int_y^{s_{j-1}} w_j(s_j)\dif s_j\cdots \dif s_2,\qquad j=2,\ldots,k.
\end{equation}
Then, the functions $g_j^w := \charfun_{x\geq y}(x,y) h_j^w(x,y)$,
$j=1,\ldots,k$, are the analogues of the truncated power functions
$(x-y)_+^{j-1}$ for polynomials, which we get by choosing the weight functions
$w_i=1$.

Let $s_1\leq s_2\leq \cdots \leq s_k$ be an increasing sequence of real numbers.
We set
	$d_i := \max \{ 0\leq j\leq k-1 : s_i = \cdots = s_{i-j} \}$
and define the expression
\begin{equation}\label{eq:defD}
	D\begin{pmatrix}
		s_1,\ldots,s_k \\
		u_1,\ldots,u_k\end{pmatrix} := \det \big( D^{d_i} u_j (t_i)
	\big)_{i,j=1}^k,
\end{equation}
where the letter $D$ on the right hand side denotes the ordinary differential operator.
Now define the sequence $(t_i)_{i=0}^n$ of grid points of the interval
$\sigma$-algebra $\mathcal F$ to be the increasingly ordered sequence of
boundary points of atoms of $\mathcal F$, where the points $0$ and $1$ each
appear $k$ times and every other point appears once in the sequence
$(t_i)_{i=0}^n$.
Then,  the Chebyshevian B-spline function $M_i^w$ for the weights $w=
(w_1,\ldots,w_k)$ and for $x\in [0,1]$ is defined by
	\begin{equation}\label{def:M}
	 M_i^{w} (x) = (-1)^k
	 \frac{	D\begin{pmatrix}
		t_i,\ldots,t_{i+k} \\
		u_1^*,\ldots,u_{k}^*,g_k^w(x,\cdot)\end{pmatrix}}{	D\begin{pmatrix}
		t_i,\ldots,t_{i+k} \\
		u_1^*,\ldots,u_{k+1}^*\end{pmatrix}},\qquad
	i=0,\ldots, n-k.
\end{equation}
The system of functions $(M_i^w)_{i=0}^{n-k}$ forms an
algebraic basis of the Chebyshevian spline space $S_{\mathcal F,w}$ and each
function $M_i^w$ has the properties
\begin{equation}\label{eq: Chebyshev L^1}
	M_i^w>0 \text{ on } (t_i,t_{i+k}),\qquad M_i^w=0 \text{ on }
	(t_i,t_{i+k})^c, \qquad \int_0^1 M_i^w(x)\dif x = 1.
\end{equation}
If we choose the weights $w_1=\cdots=w_k=1$, the above definition yields
the corresponding polynomial B-spline functions $M_i$, normalized in $L^1$ for
which we have the pointwise estimate $M_i(x) \leq k/(t_{i+k} - t_i)$.

In order to show properties (1)--(3) on page \pageref{it:perturb} for the
Chebyshevian spline spaces
$(S_{\mathcal F,w})_{\mathcal F}$ and the corresponding B-spline functions
$(M_i^w)$, we need the following result about the difference between $M_i^w$ and
the classical B-spline function $M_i$:

\begin{prop}\label{prop: Chebyshevian - classical}
 Let $M>0$ be a constant such that $1/M\leq w_i\leq M$ for $i=1,\ldots,k$.

Then there exists a constant $C>0$ depending only on $M$ and $k$ so that   the
following pointwise estimates are true:
\begin{align*}|M_i^w(x)-M_i(x)|  
		&\leq  \frac{ C}{|\supp M_i|} \max\limits_{j=1,\ldots,k}
		\omega(w_j,|\supp M_i|), \\
		|M_i^w(x) |  &\leq  \frac{ C}{|\supp M_i|},
	\end{align*}
	where $\omega(f,\delta)=\sup_{|x-y|<\delta} |f(x)-f(y)|$, as before, denotes the modulus 
	of	continuity of $f$.
\end{prop}
\begin{proof}
Clearly the second estimate is a consequence of the first estimate and the inequality $M_i(x) \leq k/(t_{i+k} - t_i)$.

	In order to estimate the difference between the Chebyshevian B-spline
	function $M_i^w$ and the classical B-spline function $M_i$, we separately estimate numerator and denominator in \eqref{def:M}.
	First we perform determinant rules to $D\begin{pmatrix}
		t_i,\ldots,t_{i+k} \\
		u_1^*,\ldots,u_{k+1}^*\end{pmatrix}$; as the first column in
	this matrix consists entirely of $1$-entries, we replace, for
	$i=2,\ldots,k+1$, the $i$th row by the difference between the $i$th 	    and the $(i-1)$st row. Recalling the corresponding definitions and factoring 	out $w_k(s_1)\cdots w_k(s_k)$ by multilinearity of the determinant, we obtain
	\[
		D\begin{pmatrix}
		t_i,\ldots,t_{i+k} \\
		u_1^*,\ldots,u_{k+1}^*\end{pmatrix} = \int_{t_i}^{t_{i+1}}\cdots
	\int_{t_{i+k-1}}^{t_{i+k}} w_k(s_1)\cdots w_k(s_k) D\begin{pmatrix}
		s_1,\ldots,s_{k} \\
		u_{1,1}^*,\ldots,u_{1,k}^*\end{pmatrix}\dif s_k\cdots\dif s_1.
	\]
	Denote by $p_i^*(t)=t^{i-1}/(i-1)!$ the function $u_i^*$ corresponding to
	the choice of weight 
	functions $w_1=\cdots=w_k=1.$
	By induction on $k$ we infer the estimate
	\begin{equation}\label{eq:ratio D}
		\prod_{j=1}^k\left(\min_{t\in[t_i,t_{i+k}]}w_j(t)\right)^j 
		\leq \frac{D\begin{pmatrix}
		t_i,\ldots,t_{i+k} \\
		u_1^*,\ldots,u_{k+1}^*\end{pmatrix}}
		{D\begin{pmatrix}
		t_i,\ldots,t_{i+k} \\
		p_1^*,\ldots,p_{k+1}^*\end{pmatrix}}
	    \leq \prod_{j=1}^k\left(\max_{t\in[t_i,t_{i+k}]}w_j(t)\right)^j .
	\end{equation} 
	Note that $D\begin{pmatrix}
		t_i,\ldots,t_{i+k} \\
		p_1^*,\ldots,p_{k+1}^*\end{pmatrix}$ is a constant multiple of the Vandermonde determinant, i.e. $D\begin{pmatrix}
		t_i,\ldots,t_{i+k} \\
		p_1^*,\ldots,p_{k+1}^*\end{pmatrix}=c \prod_{i \leq r < s
	\leq i+k} (t_s- t_r),$ where $c$ depends only on $k$.		
	Since $M^{-1} \leq w_i\leq M$ for $i=1,\ldots,k$, by \eqref{eq:ratio D}
	there exists a constant $C$
	depending only on $M$ and $k$ so that 
	\begin{equation}\label{eq:denom Vandermonde type}	
		C^{-1}\prod_{i \leq r < s
	\leq i+k} (t_s- t_r)
 \leq D\begin{pmatrix}
		t_i,\ldots,t_{i+k} \\
		u_1^*,\ldots,u_{k+1}^*\end{pmatrix} \leq C \prod_{i \leq r < s
	\leq i+k} (t_s- t_r),
	\end{equation}
	Denote 
	$$
	q_k=D\begin{pmatrix}
		t_i,\ldots,t_{i+k} \\
		\overline{u}_1^*,\ldots,\overline{u}_{k+1}^*\end{pmatrix},\qquad
		\varepsilon_k=D\begin{pmatrix}
		t_i,\ldots,t_{i+k} \\
		u_1^*,\ldots,u_{k+1}^*\end{pmatrix}-q_k,
		$$
		 where $\overline{u}_1^*,\ldots,\overline{u}_{k+1}^*$ correspond to the choice of constant weights $\overline{w}_j=\min_{t\in[t_i,t_{i+k}]} w_j(t)$.   
		 Thus, \eqref{eq:ratio D} implies
	\begin{equation}\label{eq: denom epsilon}
			0\leq \varepsilon_k \leq  q_k\cdot\left(\prod_{j=1}^k\left(1+\frac{\omega(w_j,
		t_{i+k}-t_i)}{\min_{t\in[t_i,t_{i+k}]}w_j(t)}\right)^j -1\right)
		\leq Cq_k\max_{j=1,\ldots,k} \omega(w_j, t_{i+k}-t_i), 
	\end{equation}
	for some constant $C$ depending only on $M$ and $k$. 
	
	Similarly, we estimate the numerator $	D\begin{pmatrix}
		t_i,\ldots,t_{i+k} \\
		u_1^*,\ldots,u_{k}^*,g_k^w(x,\cdot)\end{pmatrix}$ in
	\eqref{def:M}. To this end, observe that the definition
	\eqref{def:gh} of $g^w_j$ yields
	\[
		\frac{\partial}{\partial y} g_j^w(x;y) =
		-w_j(y)g_{j-1}^w(x;y),\qquad j\geq 2, x\neq y.
	\]
	Therefore, performing the same determinant rules as above, we write
	\begin{align*}
		D&\begin{pmatrix} 
		t_i,\ldots,t_{i+k} \\
		u_1^*,\ldots,u_{k}^*,g_k^w(x,\cdot)\end{pmatrix}  \\
	&=
	-\int_{t_i}^{t_{i+1}}\cdots \int_{t_{i+k-1}}^{t_{i+k}} w_k(s_1)\cdots
	w_k(s_k)
		D\begin{pmatrix}
		s_1,\ldots,s_k\\
		u_{1,1}^*,\ldots,u_{1,k-1}^*,g_{k-1}^w(x,\cdot)\end{pmatrix}\dif
	s_1\cdots \dif s_k.
	\end{align*}
	Using this formula,  induction on $k$ yields (for $x\in
	(t_i,t_{i+k})$)
	\begin{equation}\label{eq:ratio D g_k}
		\prod_{j=1}^k\left(\min_{t\in[t_i,t_{i+k}]}w_j(t)\right)^j 
		\leq \frac{   D\begin{pmatrix} 
		t_i,\ldots,t_{i+k} \\
		u_1^*,\ldots,u_{k}^*,g_k^w(x,\cdot)\end{pmatrix} }
		{ D\begin{pmatrix}
		t_i,\ldots,t_{i+k} \\
		p_1^*,\ldots,p_{k}^*, g_k(x,\cdot)\end{pmatrix}}
	    \leq \prod_{j=1}^k\left(\max_{t\in[t_i,t_{i+k}]}w_j(t)\right)^j, 
	\end{equation} 
	where $g_k(x,\cdot)$ denotes the function $g_k^w(x,\cdot)$ corresponding to the choice of the weight functions $w_i=1.$	
Define 
	$$
	r_k(x)=D\begin{pmatrix}
		t_i,\ldots,t_{i+k} \\
		\overline{u}_1^*,\ldots,\overline{u}_{k}^*, \overline{g}_k(x,\cdot)\end{pmatrix},\qquad
		\delta_k(x)=D\begin{pmatrix}
		t_i,\ldots,t_{i+k} \\
		u_1^*,\ldots,u_{k}^*, g_k(x,\cdot)\end{pmatrix}-r_k(x),
		$$
		 where 
		  $\overline{g}_k(x,\cdot)$ corresponds to the choice of constant weights $\overline{w}_j=\min_{t\in[t_i,t_{i+k}]} w_j(t)$.  	
		  Thus, using \eqref{eq:ratio D g_k} implies
	  \begin{equation}
		\label{eq:num delta} 
		\begin{aligned}
			 |\delta_k(x)| &
			 \leq  (-1)^k r_k(x)\cdot\left(\prod_{j=1}^k
			 \left(1+\frac{\omega(w_j, t_{i+k}-t_i)}
			 {\min_{t\in[t_i,t_{i+k}]}w_j(t)}\right)^j -1\right)  \\
			&\leq C(-1)^kr_k(x)\max_{j=1,\ldots,k} \omega(w_j, t_{i+k}-t_i),
		\end{aligned}
	\end{equation}
where the constant $C$ depends only on $M$ and $k$. 

Note that the ordinary B-spline $M_i(x)$ satisfies 
	 $$\frac{r_k(x)}{q_k}=
		\frac{D\begin{pmatrix}
		t_i,\ldots,t_{i+k} \\
		\overline{u}_1^*,\ldots,\overline{u}_{k}^*, \overline{g}_k(x,\cdot)\end{pmatrix}}{D\begin{pmatrix}
		t_i,\ldots,t_{i+k} \\
		\overline{u}_1^*,\ldots,\overline{u}_{k+1}^*\end{pmatrix}}
		=  
	 \frac{	D\begin{pmatrix}
		t_i,\ldots,t_{i+k} \\
		1,\ldots,t^{k-1},(x-t)^{k-1}_+\end{pmatrix}}{	D\begin{pmatrix}
		t_i,\ldots,t_{i+k} \\
		1,\ldots,t^{k}\end{pmatrix}}=  (-1)^k M_i(x).$$ 
From the pointwise estimate $M_i(x)\leq k/(t_{i+k} - t_i)$
we therefore get 
\begin{equation}\label{eq: num r_k}
|r_k(x)|=(-1)^kr_k(x)\leq k (t_{i+k} - t_i)^{-1}q_k.
\end{equation} 
Hence, summarizing the above estimates \eqref{eq: denom epsilon}, \eqref{eq:num
delta} and \eqref{eq: num r_k},
 \begin{align*}
		|M_i^w(x)-M_i(x)| &= \Big|\frac{ r_k(x) + \delta_k(x)}{q_k+\varepsilon_k} -
		\frac{r_k(x)}{q_k}\Big| = \Big| \frac{\delta_k(x) q_k -
		\varepsilon_k r_k(x)}{q_k(q_k+\varepsilon_k)} \Big|  \\
		&\leq \frac{|\delta_k(x) q_k| +
		\varepsilon_k |r_k(x)|}{q_k^2} 
		\leq C(t_{i+k}-t_i)^{-1} \max_{j=1,\ldots,k}
		\omega(w_j,t_{i+k}-t_i),
	\end{align*}
	which concludes the proof of Proposition \ref{prop: Chebyshevian - classical}.
\end{proof}

Now we show properties (1)--(3) on page \pageref{it:perturb} with the functions
$M_i^p = M_i^w$ and Lebesgue measure $\mu$. It follows from \eqref{eq: Chebyshev L^1} that $\supp_{\mathcal F} M^p_i=\supp M^p_i$. 
By Proposition~\ref{prop: Chebyshevian - classical} we have the
	following estimate 
	\begin{align*}
		|{\mu(\supp M_{j}^p)} & 
			{ \langle M_i^p, M_j^p\rangle_\mu} - {|\supp M_j|} {\langle
				M_i,M_j\rangle}| \\
				&=|\supp M_{j}|\cdot   |\langle
		M_i^p - M_i, M_j^p\rangle + \langle M_i, M_j^p - M_j\rangle|\\
		& \leq |\supp M_{i}| |\supp M_{j}|(\|M_i^p - M_i\|_\infty\| M_j^p\|_\infty+\|M_i\|_\infty\| M_j^p - M_j\|_\infty)\\
		&\leq C \max_{i=1,\ldots,k} \omega(w_i,k |\mathcal F|),
	\end{align*}
	where the constant $C$ depends only on $M$ and $k$. This confirms property
	(1). Moreover,  property (2)  
	 is satisfied by definition and (3) is a consequence of
	 Proposition~\ref{prop: Chebyshevian - classical}.
	 Moreover the
	spaces $(S_{\mathcal F,w})_{\mathcal F}$ are compatible by the
	definition of the Chebyshevian B-spline functions $M_i^w$.
	Therefore, an application of Theorem \ref{thm:perturbation} yields 
	\begin{cor}
			Let $k$ be a positive integer and
		suppose that $w=(w_1,\ldots w_k)$ is a vector of  functions on
		$[0,1]$ with
		$w_i\in C^{k-i+1}[0,1]$ 
		 satisfying the inequalities $M^{-1} \leq w_i\leq M$ for
		any $i=1,\ldots,k$.
		For an interval $\sigma$-algebra $\mathcal F$, let 
		$S_{\mathcal F,w}$ be the corresponding space of Chebyshevian
		splines and 
		$P_{\mathcal F,w}$ the orthogonal projection operator onto $S_{\mathcal F,w}$ with respect to the
	inner product $\langle \cdot,\cdot\rangle$.

		 Then, there exists a constant $C$, depending only on $k$, $M$ and the moduli
		of continuity of the weight functions $w_1,\ldots,w_k$, so that
		\[
			\sup_{\mathcal F}\| P_{\mathcal F,w} : L^\infty\to L^\infty
			\| \leq C,
		\]
		where $\sup$ is taken over all interval $\sigma$-algebras
	$\mathcal F$.
	\end{cor}
	
\subsection*{Acknowledgements } K.~K. is supported by SCS RA grant 18T-1A074 and
M.~P. is  supported by the Austrian Science Fund FWF, project P32342.

 \bibliographystyle{plain}
\bibliography{projection}

\end{document}